\documentclass[a4paper]{amsart}

\usepackage[top=3cm, bottom=3cm, left=3cm, right=3cm]{geometry}
\usepackage{amsmath,amsfonts,amssymb,amscd,bm,amsthm}
\usepackage{mathrsfs}
\usepackage[english]{babel}
\usepackage[utf8]{inputenc}
\usepackage{adjustbox}
\usepackage{graphicx}
\usepackage{float}
\usepackage{mathrsfs}
\usepackage{epsfig}
\usepackage{caption}
\usepackage{tikz}
\usepackage{hyperref}
\usepackage{xcolor}
\usepackage{orcidlink}
    \usepackage{mathtools} 

    \usepackage{tikz-cd} 
    
    \usepackage{etoolbox}

    \usepackage{appendix} 
    
    \usepackage{hyperref} 
    
    \usepackage{graphicx} 
    \usepackage{subcaption}
    
    \usepackage{xcolor} 
    
    \usepackage{tikz} 

    \usepackage{url} 
    
    \usepackage{nameref} 
    \usepackage{fancyhdr} 

    \usepackage[shortlabels]{enumitem} 
    \usepackage{float}

    \usepackage{verbatim}

\usepackage{lipsum}
\newlength\mylen
\newlist{mycases}{enumerate}{1}
\setlist[mycases,1]{label=\textbf{Case~\arabic*.}, 
  labelwidth=\dimexpr-\mylen-\labelsep\relax,leftmargin=0pt,align=right}

\usepackage[style=numeric]{biblatex}

\DeclareFieldFormat[article]{title}{#1}
\DeclareFieldFormat[book]{title}{#1}
\DeclareFieldFormat[incollection]{title}{#1}
\DeclareFieldFormat[misc]{title}{#1}
\DeclareFieldFormat[phdthesis]{title}{#1}
\DeclareFieldFormat[techreport]{title}{#1}

\renewbibmacro*{in:}{}
\DeclareNameAlias{author}{last-first}

\usepackage{hyperref}
\hypersetup{colorlinks,citecolor=black,filecolor=black,linkcolor=black,urlcolor=black}

\theoremstyle{plain}
\newtheorem{theorem}{Theorem}[section]
\newtheorem{proposition}[theorem]{Proposition}
\newtheorem{lemma}[theorem]{Lemma}
\newtheorem{cor}[theorem]{Corollary}
\theoremstyle{definition}

\newtheorem{claim}[theorem]{Claim}

\newtheorem{remark}[theorem]{Remark}
\newtheorem{example}[theorem]{Example}

\newtheorem*{theorem**}{Theorem\theoremnum}
\newenvironment{theorem*}[1][]{%
  \edef\theoremnum{\if\relax\detokenize{#1}\relax\else~#1\fi}
  \begin{theorem**}
}{%
  \end{theorem**}
}  

\newtheorem*{proposition**}{Proposition\theoremnum}
\newenvironment{proposition*}[1][]{%
  \edef\theoremnum{\if\relax\detokenize{#1}\relax\else~#1\fi}
  \begin{proposition**}
}{%
  \end{proposition**}
}  

\newtheorem*{lemma**}{Lemma\theoremnum}
\newenvironment{lemma*}[1][]{%
  \edef\theoremnum{\if\relax\detokenize{#1}\relax\else~#1\fi}
  \begin{lemma**}
}{%
  \end{lemma**}
}  

\newtheorem*{cor**}{Corollary\theoremnum}
\newenvironment{cor*}[1][]{%
  \edef\theoremnum{\if\relax\detokenize{#1}\relax\else~#1\fi}
  \begin{cor**}
}{%
  \end{cor**}
}

\newcommand{\addQEDstyle}[2]{\AtBeginEnvironment{#1}{\pushQED{\qed}\renewcommand{\qedsymbol}{#2}}\AtEndEnvironment{#1}{\popQED}}

\addQEDstyle{example}{$\triangle$}

\DeclareMathOperator{\id}{id}

\DeclareMathOperator{\Fr}{Fr}

\DeclareMathOperator{\CAT}{CAT(0)}

\DeclareMathOperator{\supp}{supp}

\DeclareMathOperator{\amod}{\mathfrak{mod}}

\newcommand{\Z}{\mathbb{Z}}

\newcommand{\Sph}{\mathbb{S}}

\newcommand{\Cp}{\mathscr{C}}
\newcommand{\Sp}{\mathscr{S}}
\newcommand{\Dp}{\mathscr{D}}

\DeclareMathOperator{\I}{I}

\usepackage{biblatex}
\usepackage{newpxtext,newpxmath}

\bibliography{bib}

\title{CAT(0) cube complexes and asymptotically rigid mapping class groups}
\author{Marie Abadie \hspace{0.1 cm}\orcidlink{0000-0002-8951-4643}}
\email{marie.abadie@uni.lu}
\date{November 2024}

\begin{document}
\maketitle
\vspace{-1.0cm}
\begin{abstract}
The present paper contributes to the study of asymptotically rigid mapping class groups of infinitely-punctured surfaces obtained by thickening planar trees. In a paper from 2022, Genevois, Lonjou and Urech a study of the latter groups by using cube complexes. We determine in which cases their cube complexes are $\CAT$. From this study, we develop a family of $\CAT$ cubes complexes on which the asymptotically rigid mapping class groups act.
\end{abstract}

\section{Introduction}
Thompson's groups have been the subject of intense study in group theory and have proven to be a rich source of interesting examples \cite{Belk, Thoms,Shavgulidze2009THETG}. They inspired the construction of other groups through variations of their concepts, called \emph{Thompson-like groups}, \cite{higman1974finitely, Houghtonpres}. A number of recent work has been devoted to the study of braided versions of Thompson-like groups \cite{Brin_2007, dehornoy} i.e. extensions of Thompson-like groups by infinite braid groups, which turn out to be closely linked to mapping class groups of surfaces of infinite type \cite{aramayona2020big, Funar3}. 

In \cite{asymI}, Genevois, Lonjou and Urech study a particular family of braided Thompson-like groups called \emph{asymptotically rigid mapping class groups}. Their framework is inspired from \cite{Funar4, Funar2011AsymptoticallyRM}. These groups are subgroups of \emph{big mapping class groups} of infinitely-punctured surfaces obtained by thickening planar trees $A_{n,m}$, having one vertex of valence $m$ while all the others have valence $n+1$. Briefly, we give a rigid structure $\Sp^\sharp(A_{n,m})$ to the latter surfaces and consider the group of isotopy classes of homeomorphisms that preserve this structure 'almost everywhere'. To investigate the finiteness properties of these groups, denoted by $\amod(A_{n,m})$, the aforementioned authors present a novel family of cube complexes called $\Cp(A_{n,m})$. They raise the question of whether their cube complexes are non-positively curved \cite[Section 3.3]{asymI}. The answer to this question is the following result (Theorem \ref{Prop1} in the text):

\begin{theorem*}
The cube complex $\Cp(A_{n,m})$ is $\CAT$ if and only if $1 \leq m \leq n+1$.
\end{theorem*}

To study the case where $m\geq n$, we consider different rigid structures $\Sp^*(A_{n,m})$ which are related to the previous one through the following observation (Lemma \ref{LemmaIsoStruct}):
\begin{lemma*}
  For all $m,n \geq 1$, there is a group isomorphism \[\amod(\Sp^{\sharp}(A_{n,m})) \cong \amod(\Sp^*(A_{n,m-n+1})).\]
\end{lemma*}

In \cite{asymII}, the authors suggested that it may be possible to modify their construction. Hence, we introduce a cube complex $\Dp(A_{n,m})$ that is $\CAT$ for all $m,n\geq 1$, and upon which the asymptotically rigid mapping class groups of $\Sp^*(A_{n,m})$ act (Theorem \ref{DpisCAT0} and Corollary \ref{EndCor} in the text). 

\begin{theorem*}
     For all $m,n \geq 1$ the cube complex $\Dp(A_{n,m})$ is $\CAT$.
\end{theorem*}

Thus, we define a collection of $\CAT$ cube complexes,
$$
\mathcal{E}(A_{n,m}):= \left\{
    \begin{array}{ll}
        \Dp(A_{n,m-n+1}) & \mbox{if } m > n+1 \geq 1 \\
        \Cp(A_{n,m}) & \mbox{if } 1 \leq m \leq n+1
    \end{array}.
\right.
$$

\begin{cor*}
     For all $m,n \geq 1$, $\amod(A_{n,m})$ acts on the cube complex $\mathcal{E}(A_{n,m})$ which is $\CAT$.
\end{cor*}

\subsection*{Acknowledgements}
The main part of this work was elaborated in the context of a Master thesis under the supervision of Christian Urech who provided the main question, guidance and a friendly environment. We thank the referee for the comments and suggestions that helped us to improve the clarity of the present paper. This work was partially supported by the Luxembourg National Research Fund OPEN grant O19/13865598.

\section{Preliminaries}
Let $A$ be a locally finite planar tree. Its \emph{arboreal surface} is the surface $\Sp(A)$ obtained by embedding A into the plane and thickening it. Let $\Sp^\sharp(A)$ be the surface obtained from $\Sp(A)$ by adding a puncture for each vertex of the underlying tree $A$. We give a \emph{rigid structure} to $\Sp^\sharp(A)$ by dividing the surface into \emph{polygons} with a family of pairwise non-intersecting arcs whose endpoints are on the boundary of the surface such that: \begin{itemize}
        \item each arc intersects one unique edge of $A$ and this intersection is transverse,
        \item each polygon contains exactly one puncture in its interior.
    \end{itemize}
\begin{figure}[h]
    \centering
    \includegraphics[scale=0.18]{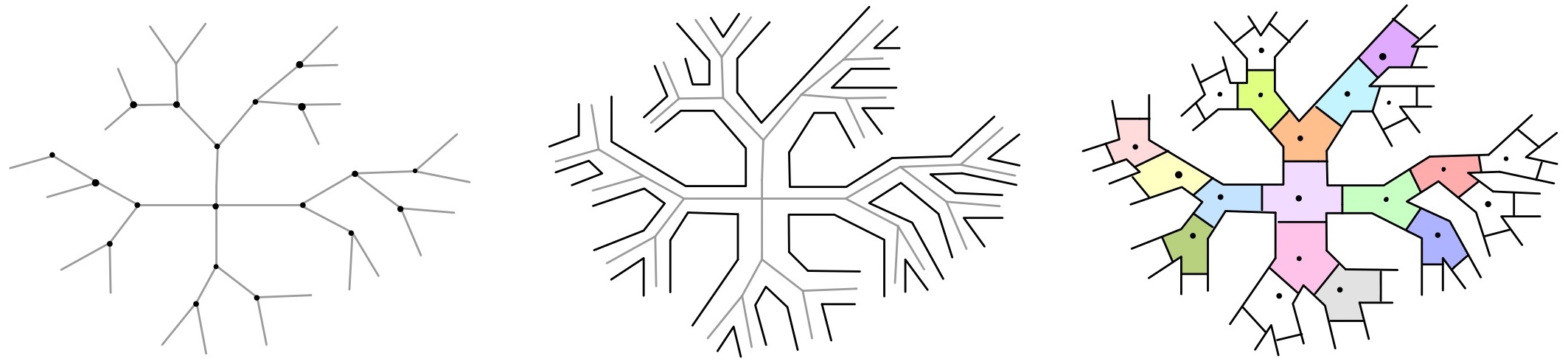}
    \caption{From left to right, pictures of $A_{2,4}$, $\Sp(A_{2,4})$ and $\Sp^\sharp(A_{2,4})$.}
    \label{Figure_00}
\end{figure} 

An \emph{admissible} subsurface $\Sigma \subseteq \Sp^\sharp(A)$ is defined as a connected subsurface that can be written as a finite union of polygons from the rigid structure. The \emph{height} of $\Sigma$, denoted as $h(\Sigma)$, refers to the number of punctures present within $\Sigma$. The  \emph{frontier} of $\Sigma$ refers to the union of all the arcs, called \emph{frontier arcs}, from the rigid structure lying in the boundary of $\Sigma$. It is denoted as $\Fr(\Sigma)$. A polygon $H$ is \emph{adjacent} to $\Sigma$ if it shares an arc with the frontier of $\Sigma$.
\medskip

An \emph{asymptotically rigid homeomorphism} of $\Sp^\sharp(A)$ is a homeomorphism $\phi$ from $\Sp^\sharp(A)$ to itself which respects the rigid structure almost everywhere. That is, there exists an admissible surface $\supp_{\phi}$, called the \emph{support} of $\phi$, such that:\begin{itemize}
        \item its image $\phi(\supp_{\phi})$ is an admissible surface,
        \item outside its support $\supp_{\phi}$, $\phi$ sends polygons to polygons.
    \end{itemize}

The group of isotopy classes of orientation-preserving asymptotically rigid homeomorphisms of $\Sp^\sharp(A)$ is denoted by $\amod(A)$ or $\amod(\Sp^\sharp(A))$ and is called the \emph{asymptotically rigid mapping class group} associated to $A$. Let $\phi \in \amod(A)$ and $\Sigma$ an admissible surface, $\phi$ is \emph{rigid} outside $\Sigma$ if each polygon of the rigid structure not in $\Sigma$ is mapped to a polygon. 

Funar and Kapoudjian were the first to consider the group corresponding to a regular tree of degree three and to study its generators and relations \cite{Funar4}. In \cite{asymI}, Genevois, Lonjou and Urech extend the definition to explore the finiteness properties of these groups when the tree $A$ is considered to be $A_{n,m}$.
\medskip

We define an equivalence relation on the set of pairs $(\Sigma, \phi)$, where $\Sigma$ is an admissible surface of $\Sp^\sharp(A)$ and $\phi$ lies in $\amod(A)$. Two pairs are related  $(\Sigma, \phi) \sim (\Sigma', \phi')$ if $\phi'^{-1} \phi$ is isotopic to an asymptotically rigid homeomorphism that is rigid outside $\Sigma$ and maps $\Sigma$ to $\Sigma'$. Observe that $(\Sigma, \phi) \sim ~(\Sigma, \phi')$ if $\phi$ is isotopic to $\phi'$. We denote the equivalence class of the pair $(\Sigma, \phi)$ by $[\Sigma, \phi]$.

In \cite{asymI}, the authors introduce a new family of cube complexes to explore the finiteness properties of this group when the tree $A$ is considered to be $A_{n,m}$. They define the cube complex $\Cp(A_{n,m})$ as follows: \begin{itemize}[-]
     \item \underline{vertices} $[\Sigma, \phi]$ for each admissible surface $\Sigma$ of $\Sp^\sharp(A_{n,m})$ and each $\phi\in\amod(A_{n,m})$;
     \item \underline{edges} between any two vertices of the form $[\Sigma, \phi]$ and $[\Sigma \cup H, \phi]$, where $H$ is a polygon adjacent to $\Sigma$;
     \item \underline{$k$-cubes} with underlying subgraphs of the form $\left\{ [\Sigma \cup {\bigcup}_{i \in I} H_i, \phi] \; |\; I \subset \{1,\dots,k\}\right\}$ where $[\Sigma, \phi]$ is a vertex and $H_1,\dots,H_k$ are distinct adjacent polygons to $\Sigma$.
 \end{itemize}

Let $g \in \amod(A_{n,m})$ and $[\Sigma, \phi]$ be a vertex in $\Cp(A_{n,m})$, we define: $g \cdot [\Sigma, \phi] := [\Sigma, g\circ \phi].$ This gives an action of $\amod(A_{n,m})$ on $\Cp(A_{n,m})$.
\medskip

Now, let us vary the rigid structure associated to $\Sp(A)$. Let $\Sp^*(A_{n,m})$ denote the punctured arboreal surface obtained from $\Sp(A_{n,m})$ by adding a puncture for each vertex of the tree excepting the vertex of valence $m$. We define another rigid structure on $\Sp^*(A_{n,m})$. To this end, we divide the surface into polygons with a family of pairwise non-intersecting arcs whose endpoints are on the boundary of the surface such that: \begin{itemize}
     \item each arc crosses once and transversely a unique edge of the tree,
     \item each polygon contains exactly one vertex of the underlying tree in its interior. Since each vertex corresponds to a puncture except the $m$-valence one, each polygon contains a puncture except the central one. 
 \end{itemize} 

The definitions of admissible surface, height, and frontier arc remain the same as for $\Sp^\sharp(A_{n,m})$. It is worth noting that the central polygon has height zero in this new formalism. An \emph{asymptotically rigid homeomorphism} of $\Sp^*(A_{n,m})$ is a homeomorphism $\phi$ from $\Sp^*(A_{n,m})$ to itself which respects the rigid structure almost everywhere. That is, there exists a minimal admissible surface called the \emph{support} of $\phi$, denoted $\supp^*_{\phi}$ such that:\begin{itemize}
        \item its image $\phi(\supp^*_{\phi})$ is an admissible surface,
        \item $\phi$ sends polygons to polygons outside $\supp^*_{\phi}$.
    \end{itemize}
We call $\amod(\Sp^*(A_{n,m}))$ the group of isotopy classes of orientation-preserving asymptotically rigid homeomorphisms of $\Sp^*(A_{n,m})$. To simplify the notation we set $\amod^*(A_{n,m}) := \amod(\Sp^*(A_{n,m}))$. 
\medskip

 \begin{lemma}\label{LemmaIsoStruct}
     There is a group isomorphism \[\amod(A_{n,m}) \cong \amod^*(A_{n,m-n+1}) \]
 \end{lemma}
\begin{proof}
    The following argument follows closely the proof of Lemma 2.4 in \cite{asymI}. We fix $u$ the central vertex in $A_{n,m}$ and $v$ one of its  neighbours. We let $A'_{n,m}$ be the tree obtained from $A_{n,m}$ by collapsing the edges between $u$ and $v$. Observe that $A'_{n,m}=A_{n,m+n-1}$. Let $M$ and $B$ denote the polygons of $\Sp^*(A_{n,m})$ containing $u$ and $v$, respectively. We define a new rigid structure on $\Sp^*(A_{n,m})$ by removing the arc common to $M$ and $B$, and denote it by $\Sp^\circ(A_{n,m})$. Notice that $\Sp^\circ(A_{n,m})$ coincides with $\Sp^{\sharp}(A'_{n,m})$ up to a homeomorphism that preserves the rigid structure. Figure \ref{Fig_2.09} illustrates the case $n=2$. 
    \begin{figure}[h]
    \centering
    \includegraphics[width=0.87\linewidth]{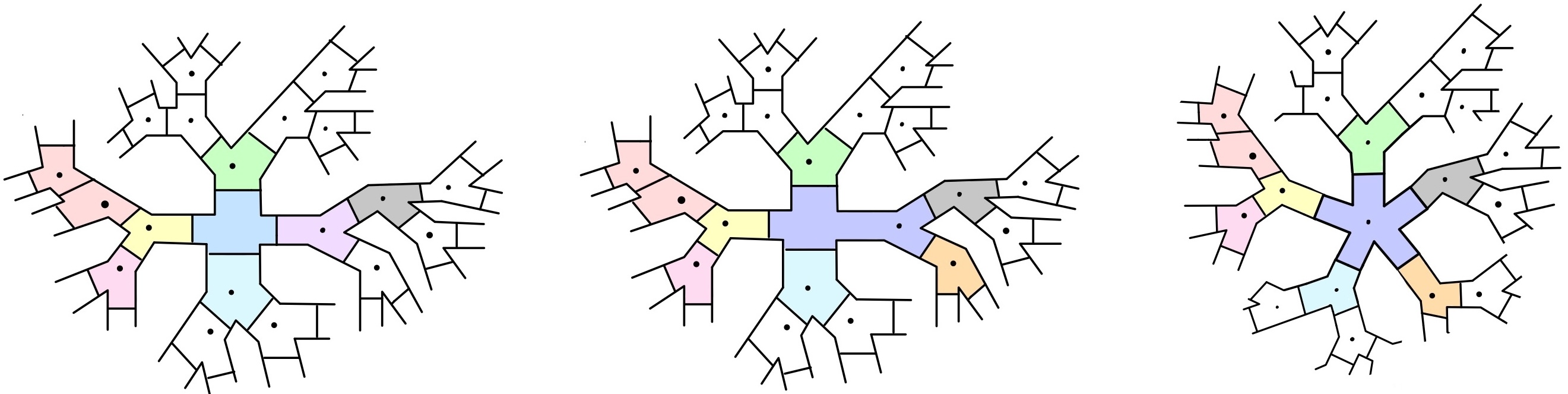}
    \caption{From left to right, pictures of $\Sp^*(A_{2,4})$, $\Sp^{\circ}(A_{2,4})$ and $\Sp^\sharp(A'_{2,4})$.}
     \label{Fig_2.09}
\end{figure}
    Therefore there exists a homeomorphism $\psi~:~\Sp^{\sharp}~(~A~'_{n,m}) ~\longrightarrow~ \Sp^*~(A_{n,m})$ that sends each polygon of $\Sp^{\sharp}(A'_{n,m})$ to a polygon of $\Sp^*(A_{n,m})$ except one that is sent to the union of two polygons. Hence, the conjugation by $\psi$ gives an isomorphism $\amod^*(A_{n,m})\cong \amod(A_{n,m+n-1})$ so $\amod^*(A_{n,m-n+1}) \cong \amod(A_{n,m})$. See Figure \ref{Figure_12.3} for the cases $n=1,2$.
    \begin{figure}[h]
    \centering
    \includegraphics[width=0.95\linewidth]{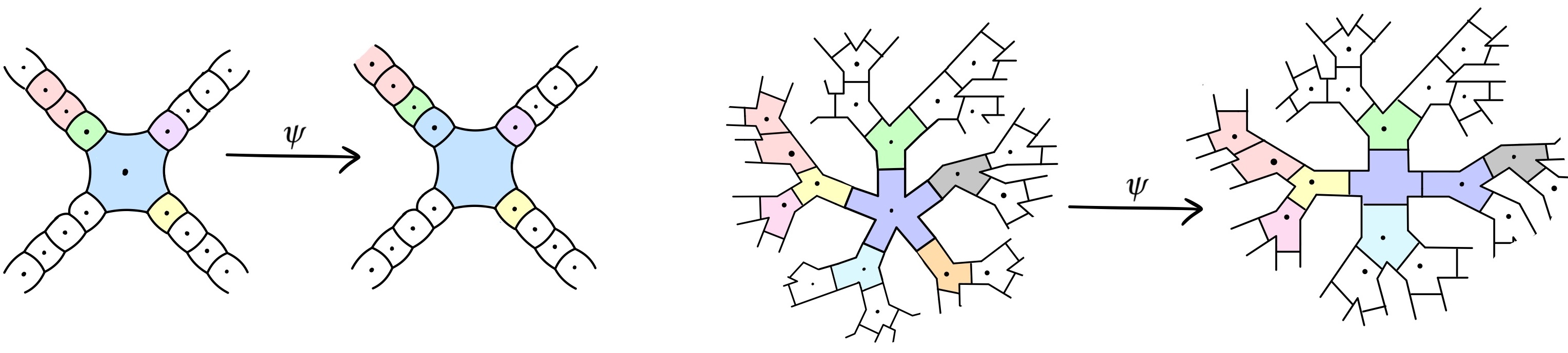}
    \caption{}
     \label{Figure_12.3}
\end{figure}

\end{proof}
In the following definition we work with the rigid structure $\Sp^*(A_{n,m})$.

Let $\Dp(A_{n,m})$ be the cube complex defined as follows: \begin{itemize}[-]
     \item \underline{vertices} $[\Sigma, \phi]$ where $\Sigma \subset \Sp^*(A_{n,m})$ is an admissible surface that contains the central polygon and $\phi\in\amod^*(A_{n,m})$;
     \item \underline{edges} between any two vertices of the form $[\Sigma, \phi]$ and $[\Sigma \cup H, \phi]$, where $H$ is a polygon adjacent to $\Sigma$;
     \item \underline{$k$-cube} with underlying subgraph of the form $\left\{ [\Sigma \cup {\bigcup}_{i \in I} H_i, \phi] \; |\; I \subset \{1,\dots,k\}\right\}$ where $[\Sigma, \phi]$ is a vertex and $H_1,\dots,H_k$ are distinct adjacent polygons to $\Sigma$.
 \end{itemize} Note that here two pairs are related $(\Sigma, \phi) \sim (\Sigma', \phi')$ if $\phi'^{-1} \phi$ is isotopic to an element of $\amod^*(A_{n,m})$ that is rigid outside $\Sigma$ and maps $\Sigma$ to $\Sigma'$. 
\medskip

Next, we present some properties about $\Cp(A_{n,m})$ and $\Dp(A_{n,m})$.

\begin{lemma}[Lemma 3.4 \cite{asymI}]\label{lemma3.4}
    Let $x=[\Sigma, \phi]$ and $y$ be two adjacent vertices of $\Cp(A_{n,m})$, resp. $\Dp(A_{n,m})$. If $h(y)>h(x)$, then there exists $H$ an adjacent polygon to $\Sigma$ such that  $y=[\Sigma \cup H, \phi]$.
\end{lemma}

\emph{The height-orientation} is obtained by giving an orientation to all the edges of  $\Cp(A_{n,m})$, resp. $\Dp(A_{n,m})$, from lower height to higher height. Given a vertex $x$, we say that an incident edge of x \emph{points outwards} if it is oriented from $x$ to another vertex; otherwise, it \emph{points towards} $x$. The \emph{height of a vertex} $x=[\Sigma, \phi]$, denoted by $h(x)$, refers to the height of the admissible surface $\Sigma$.

\begin{lemma}\label{case2cube}
    Let $S$ be a $2$-cube in $\Cp(A_{n,m})$, resp. $\Dp(A_{n,m})$. Then $S$ has a unique height-orientation. More precisely, one of the vertices has some height $h \in \Z_{>0}$, its two adjacent vertices have height $h+1$ and the remaining vertex has height $h+2$. 
\end{lemma}
\begin{proof}
      We work with $\Cp(A_{n,m})$, the proof remains exactly the same for $\Dp(A_{n,m})$. Give an orientation to each edge of $S$ with respect to the height, this brings us to consider four cases in Figure~\ref{Fig02}. \begin{figure}[H]
    \centering
    \includegraphics[scale=0.2]{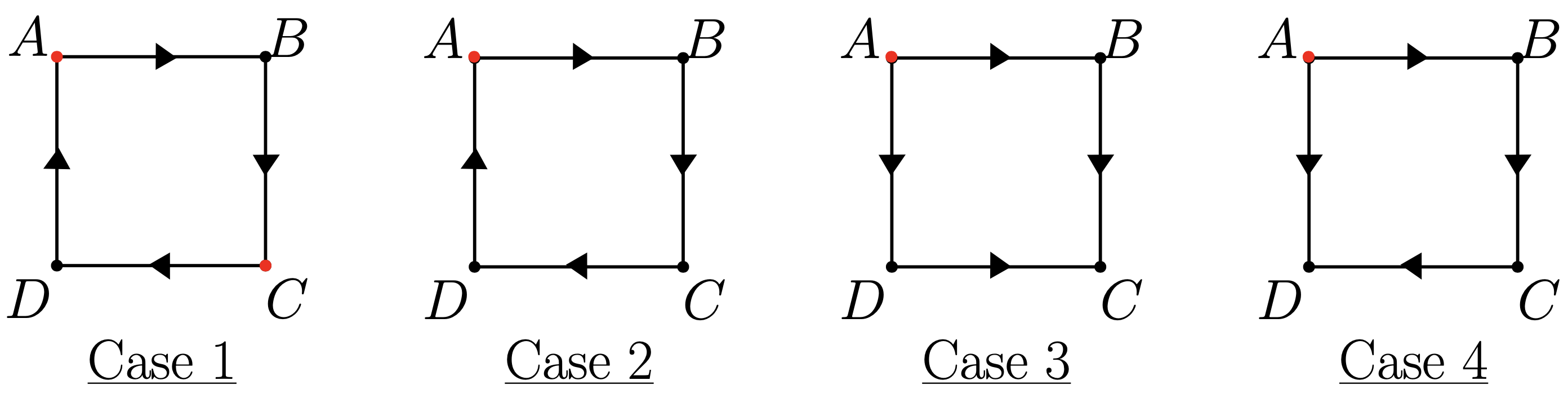}
    \caption{ }
    \label{Fig02} \end{figure}
Applying Lemma \ref{lemma3.4}, starting from the  vertex A, allows us to keep only the third case which is a square of the desired form. Indeed, the second and fourth cases are excluded because of a contradiction with the height of the vertices. Then consider the first case, we assume that the vertex $A$ (resp. $C$) is $[\Sigma,\id]$ (resp.  $[\Gamma,\pi]$). On one hand, by Lemma \ref{lemma3.4} applied from $A$ to $B$ and from $C$ to $B$, one has, $[\Sigma \cup H_1,\id]=[\Gamma \cup I_1,\pi]$ for some polygons $H_1$ and $I_1$ adjacent to $\Sigma$ and $\Gamma$ respectively. Hence, $\pi(\Gamma \cup I_1)=\Sigma \cup H_1$ and $\pi$ is rigid outside $\Gamma \cup I_1$. On the other hand, by Lemma \ref{lemma3.4} applied from $A$ to $D$ and from $C$ to $D$, one has, $[\Sigma \cup H_2,\id]=[\Gamma \cup I_2,\pi]$. By definition, $\pi(\Gamma \cup I_2)=\Sigma \cup H_2$ and $\pi$ is rigid outside $\Gamma \cup I_2$. Thus $\pi(I_i)=H_i$ for $i=1,2$. Consequently, $[\Gamma, \pi]=[\Sigma, \id]$, which is a contradiction.
\end{proof}

\begin{cor}\label{heightcube}
    Any $3$-cube must have a specific orientation with respect to the height function. In particular, it must have one smallest vertex $x$, then three vertices of height $h(x) + 1$, three vertices of height $h(x) + 2$, and one maximal vertex of height $h(x)+3$. 
\end{cor}

Let $\Gamma$ be a graph, its cube-completion $\Gamma^\square$ is the cube complex obtained from $\Gamma$ by filling every sub-graph isomorphic to the 1-skeleton of a cube with a cube of the corresponding dimension. The cube complex $\Cp(A_{n,m})$ is the cube completion of its underlying graph. This is the content of the following proposition.

\begin{proposition}\label{cubecomp}
    Let $\Gamma_{A_{n,m}}$ be the underlying graph of $\Cp(A_{n,m})$, resp. $\Dp(A_{n,m})$. Then $\Gamma_{A_{n,m}}^{\square}=\Cp(A_{n,m})$ (resp. $\Dp(A_{n,m})$).
\end{proposition}

\begin{proof}
We work with $\Cp(A_{n,m})$ but the proof remains exactly the same for $\Dp(A_{n,m})$. The inclusion $\Cp(A_{n,m}) \subset \Gamma_{A_{n,m}}^{\square}$ follows from the definitions. Let $S$ be a $n$-cube in $\Gamma_{A_{n,m}}^{\square}$. We show by induction that its vertices form an $n$-cube in $\Cp(A_{n,m})$ as well. In other words, we need to show that its vertices have the form $\{[\Sigma\cup\bigcup_{i\in I}H_i,\phi]\;\mid\;I\subset\{1,\dots,n\}\}$ for some admissible surface $\Sigma$, an asymptotically rigid homeomorphism $\phi$, and distinct polygons $H_1,\dots,H_n$ adjacent to $\Sigma$. 

\begin{claim}
    We label the vertices of $S$ by bit strings $\underline{e}=e_1\dots e_n$ with $e_i\in\{0,1\}$ so that the height function $h$ attains a mimimum at $\underline{0}=0\dots 0$ (and two adjacent vertices differ by one bit). Then, $h(\underline{e})=h(\underline{0})+\textstyle\sum_{i=1}^ne_i$ for each $\underline{e}$. In particular, the function $h$ has a unique global minimum and a unique global maximum on the vertices of $S$.
\end{claim}
\begin{proof}[Proof of the claim]
We show this by induction on $n$. For $n=1$, this is the content of Lemma \ref{lemma3.4}. We will need the base case $n=2$ (treated by Lemma \ref{case2cube}) in the proof. Observe that, along any edge, the height must increase by $1$ or decrease by $1$. Assume that it holds for any $(n-1)$-cube and consider the $n$-cube $S$.

Pick a vertex in $S$ with minimal height and give it the label $\underline{0}=0\dots 0$. Label the rest of the cube coherently. Any vertex of $S$ except $\underline{1}=1\dots 1$ shares an $(n-1)$-cube with $\underline{0}$, thus by induction hypothesis we have $h(\underline{e})=h(\underline{0})+\textstyle\sum_{i=1}^ne_i$ for all $\underline{e}\neq\underline{1}$. In particular, all the neighbours of $\underline{1}$ have height $h(\underline{0})+(n-1)$. Thus $\underline{1}$ must have height $h(\underline{0})+n$ or $h(\underline{0})+(n-2)$. The latter case cannot happen. Indeed it would mean that one can find a $2$-cube called $S'$ containing $\underline{1}$, two of its neighbours and a neighbour of them so that we would have two points with maximal height in $S'$, as in Case 1 of Figure \ref{Fig02}. Thus $\underline{1}$ has height $h(\underline{0})+n$ and we are done.
\end{proof}

Fix a labelling of the vertices of $S$ as in the above claim. Then, by Lemma \ref{lemma3.4}, the $n$ neighbours of $\underline{0}$ must correspond to the attachment of $n$ distinct polygons to $\Sigma$. More precisely, the vertex with the $1$ in the $i$-th entry has the form $[\Sigma\cup H_i,\phi]$, and the polygons $H_1,\dots,H_n$ are all distinct and adjacent to $\Sigma$.

Now, the vertices with exactly two $1$'s in their label are also determined by the above. For example, the vertex with $1$'s in entries $i$ and $j$ must have the form $[\Sigma\cup H_i\cup H_j,\phi]$. This follows from the base case $n=2$. By induction,  we are able to identify all vertices in the cube $S$ (each time using the base case $n=2$). More precisely, the vertex whose label has $1$'s in entries $i_1,\dots,i_k$ must have the form $[\Sigma\cup H_{i_1}\cup\dots\cup H_{i_k},\phi]$. This concludes the proof.
\end{proof}
     
\section{Studying the \texorpdfstring{$\CAT$}{TEXT}ness of the cube complexes \texorpdfstring{$\Cp(A_{n,m})$}%
     {TEXT} and \texorpdfstring{$\Dp(A_{n,m})$}%
     {TEXT}}
 
In the first part of this section, we derive some properties on $\Cp(A_{n,m})$ and $\Dp(A_{n,m})$ in order to deduce, in the second part, under which condition $\Cp(A_{n,m})$ is $\CAT$. Finally, in the third part we study the $\CAT$ness of $\Dp(A_{n,m})$.  

Recall that a cube complex is said to be $\CAT$ if it is simply connected and the link of every vertex is \emph{flag}. In the context of a cube complex $X$, a \emph{$1$-corner} refers to a square with two adjacent edges identified. Similarly, a \emph{$2$-corner} consists of two squares that share two consecutive edges. Hence, having a \emph{$2$-corner} implies that the $1$-skeleton contains a copy of the bipartite graph $K_{3,2}$. A \emph{$3$-corner} in $X$ consists of $3$ pairwise distinct squares that share a vertex, pairwise share an edge, and are not contained in a $3$-cube of $X$. The common vertex of these squares is referred to as the \emph{root of the $3$-corner}.

The next proposition provides a tool to show the fact that a cube complex is $\CAT$. It refers to Theorem 3.3.1 in \cite{BookMedian} and Theorem 6.1 in \cite{Chep}.

\begin{proposition}\label{Prop2}
Let $X$ be a connected graph. Assume that:\begin{enumerate}
    \item\label{hyp1} the cube completion $X^{\square}$ is simply connected,
    \item\label{hyp2}  $X$ has no $1$-corner,
    \item\label{hyp3}  $X^{\square}$ also satisfies the \emph{$3$-cube condition} (every $3$-corner must span a $3$-cube),
    \item\label{hyp4}  $X$ does not contain a copy of the complete bipartite graph $K_{3,2}$.
\end{enumerate}
Then, $X^{\square}$ is a $\CAT$ cube complex.
\end{proposition}

By definition, for all $n,m\geq 1,$ there is no $1$-corner in the $1$-skeleton of $\Cp(A_{n,m})$ and $\Dp(A_{n,m})$, respectively. 
 
\begin{lemma}\label{lemmaK23}
    There is no copy of the bipartite graph $K_{2,3}$ in the $1$-skeleton of both $\Cp(A_{n,m})$ and $\Dp(A_{n,m})$.
\end{lemma}

\begin{proof}
    We work with $\Cp(A_{n,m})$ but the proof is the same for $\Dp(A_{n,m})$. Assume by contradiction that $K_{2,3}$ exists in the $1$-skeleton of $\Cp(A_{n,m})$ and give the height-orientation to each edge. This leads us to consider eight cases in Figure \ref{Fig3.0}.
     \begin{figure}[h]
    \centering
    \includegraphics[scale=0.3]{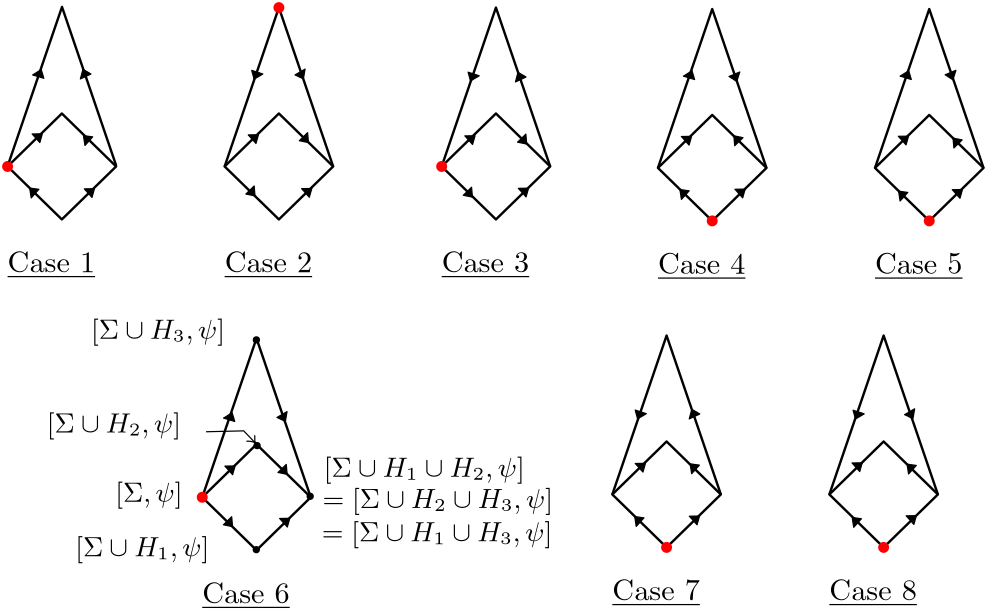}
    \caption{ }
    \label{Fig3.0} \end{figure} 
    
     By Lemma \ref{case2cube}, it remains to study Case 6. We mark with a red dot the vertices from which all incident edges point outwards. Let's assume the red point is $[\Sigma,\psi]$. By Lemma \ref{lemma3.4}, its adjacent vertices are $[\Sigma \cup H_1, \psi]$, $[\Sigma \cup H_2, \psi]$, and $[\Sigma \cup H_3, \psi]$, where $H_1$, $H_2$, and $H_3$ are some polygons adjacent to $\Sigma$. Then, by Lemma \ref{lemma3.4}, the remaining vertex, whose incident edges all point towards it, is $[\Sigma \cup H_1 \cup H_2, \psi] = [\Sigma \cup H_2 \cup H_3, \psi] = [\Sigma \cup H_1 \cup H_3, \psi]$. Consequently, $H_1=H_2=H_3$, which contradicts the distinctness of the vertices.
\end{proof}

\subsection{Completing roots of $3$-corners}
In this subsection, we study the $3$-cube condition in $\Cp(A_{n,m})$ and $\Dp(A_{n,m})$. A root of a $3$-corner is defined as an \emph{attracting root} if all its incident edges point towards it. A root with at least one incident edge pointing outwards is referred to as a \emph{non-attracting} root. The main obstacle for $\Cp(A_{n,m})$ to be $\CAT$ comes from the absence of polygons of height zero in the rigid structure. Indeed, there exists 3-corners with three vertices of height $1$, three of height $2$, and one of height $3$. However, by the Corollary \ref{heightcube} the only way to complete such a 3-corner into a 3-cube would be to add a vertex of height $0$. Hence in the absence of polygons of height zero, 3-corners exist and the $\CAT$ property cannot be verified by Proposition \ref{Prop2}. By construction $\Dp(A_{n,m})$ contains a polygon of height zero, the central polygon, which avoids the presence of 3-corners.
    
\begin{lemma}\label{lem3cornerA}
    Every $3$-corner in $\Cp(A_{n,m})$, resp. $\Dp(A_{n,m})$, having a non-attracting root can be completed into a $3$-cube. 
\end{lemma}

\begin{proof}
   We consider $\Cp(A_{n,m})$, the proof remains the same for $\Dp(A_{n,m})$. In Figure \ref{Fig111}, we illustrate the three possible $3$-corners with non-attracting roots in $\Cp(A_{n,m})$. \begin{figure}
    \centering
    \includegraphics[scale=0.3]{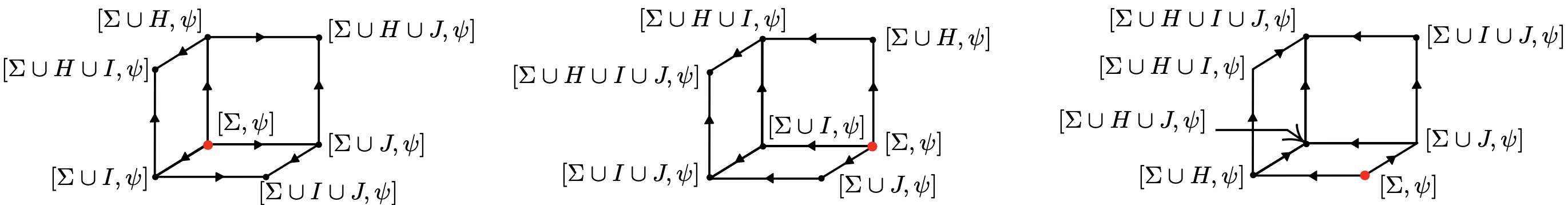}
    \caption{ }
    \label{Fig111} \end{figure}
     In each case, our objective is to identify a vertex that completes the $3$-corner into a $3$-cube. We mark with a red dot the vertices from which all incident edges point outwards.
    \begin{itemize}
         \item Consider Case 1 in Figure \ref{Fig111}. In this case, there is only one red point written as $[\Sigma, \psi]$. By Lemma \ref{lemma3.4}, there are three polygons $H,I,J$ adjacent to $\Sigma$, so that the vertices adjacent to the red point are $[\Sigma \cup H, \psi]$, $[\Sigma \cup J, \psi]$ and $[\Sigma \cup I, \psi]$. Again by Lemma \ref{lemma3.4}, the common adjacent vertices of the previous vertices are  $[\Sigma \cup H \cup J, \psi]$, $[\Sigma \cup H \cup I, \psi]$ and $[\Sigma \cup I \cup J, \psi]$. The vertex $[\Sigma \cup H \cup I\cup J, \psi]$ completes the $3$-corner into a $3$-cube.
        \item  Here we consider Case 2 in Figure \ref{Fig111}. There is only one red point which can be written as $[\Sigma, \psi]$. As for the case above, by Lemma \ref{lemma3.4}, its adjacent vertices are $[\Sigma \cup H, \psi]$, $[\Sigma \cup J, \psi]$ and $[\Sigma \cup I, \psi]$. The common adjacent vertices of the previous vertices are $[\Sigma \cup H \cup I, \psi]$ and $[\Sigma \cup I \cup J, \psi]$. The remaining vertex is $[\Sigma \cup H \cup J, \psi]$ and we take the vertex $[\Sigma \cup H \cup J, \psi]$ to complete the $3$-corner into a $3$-cube.
        \item For the third case in Figure \ref{Fig111} we proceed in the same way as above. Assume that the red point is $[\Sigma,  \psi]$. Again, by Lemma \ref{lemma3.4}, the adjacent vertices of the latter are $[\Sigma \cup H, \psi]$ and $[\Sigma \cup J, \psi]$. Consequently the root of the $3$-corner is $[\Sigma \cup H \cup J, \psi ]$. The remaining adjacent vertices of $[\Sigma \cup H, \psi]$, resp. $[\Sigma \cup J, \psi]$ are written
        $[\Sigma \cup H \cup K_1, \psi]$, resp. $[\Sigma \cup J \cup K_2, \psi]$. Next, the common adjacent vertices of $[\Sigma \cup H \cup K_1, \psi]$, $[\Sigma \cup J \cup K_2, \psi]$ and $[\Sigma \cup H \cup J, \psi]$ must have an height equals to $h(\Sigma)+3$ so $I:=K_1=K_2$. The remaining vertex, having all its edges pointing towards it is $[\Sigma \cup H \cup J \cup I, \psi]$. The vertex $[\Sigma \cup I, \phi]$ completes the $3$-corner into a $3$-cube.
    \end{itemize}      
\end{proof}

\begin{figure}[h]
    \centering
    \includegraphics[scale=0.3]{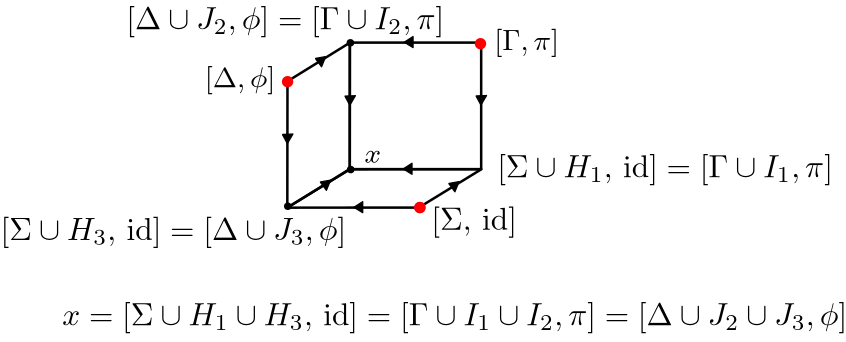}
    \caption{ }
    \label{Fig10} \end{figure}

In this section, we denote by $M$ the central polygon in $\Cp(A_{n,m})$, resp. $\Dp(A_{n,m})$, which is the polygon whose underlying vertex has valence $m$. If $m=n+1$, the central polygon in $\Cp(A_{n,m})$ is an arbitrary polygon that we fix. We use the notation in Figure~\ref{Fig10} to describe a $3$-corner in $\Cp(A_{n,m})$, resp. $\Dp(A_{n,m})$, having an attracting root.

Our last goal in this subsection is to show the following proposition:

\begin{proposition}\label{prop3corner}\label{lemma2}\label{lemma4}\label{lemma1}
Let $(m,n)\neq (1,2),\; (2,1)$, every $3$-corner with an attracting root in $\Dp(A_{n,m})$ can be completed into a $3$-cube. If $1 \leq m \leq n+1$, then this also holds in $\Cp(A_{n,m})$.
\end{proposition}

The following facts can be deduced from the equalities associated with each vertex in Figure~\ref{Fig10}:
\begin{enumerate}[label=(\roman*)]
            \item \label{i} $\pi(\Gamma \cup I_1)= \Sigma \cup H_1$ and $\pi$ is rigid outside $\Gamma \cup I_1$,
            \item \label{ii} $\phi(\Delta \cup J_3)= \Sigma \cup H_3$ and $\phi$ is rigid outside $\Delta \cup J_3$,
            \item \label{iii} $\pi^{-1}\phi(\Delta \cup J_2)= \Gamma \cup I_2$ and $\pi^{-1}\phi$ is rigid outside $\Delta \cup J_2$,
            \item \label{iv} $\pi(\Gamma \cup I_1 \cup I_2)= \Sigma \cup H_1 \cup H_3$ and $\pi$ is rigid outside $\Gamma \cup I_1 \cup I_2$,
            \item  \label{v} $\phi(\Delta \cup J_2 \cup J_3)= \Sigma \cup H_1 \cup H_3$ and $\phi$ is rigid outside $\Delta \cup J_2 \cup J_3$,
            \item \label{vi} $\pi^{-1}\phi(\Delta \cup J_2 \cup J_3)= \Gamma \cup I_1 \cup I_2$ and $\pi^{-1}\phi$ is rigid outside $\Delta \cup J_2 \cup J_3$.
        \end{enumerate}

We need the four following lemmata to prove Proposition \ref{lemma2}.

\begin{lemma}\label{fact1}
     Assume that we are in the situation of a $3$-corner with an attracting root in $\Cp(A_{n,m})$, resp. $\Dp(A_{n,m})$, as depicted in Figure~\ref{Fig10}. The image of the polygon $I_1$ by $\pi$ is contained in $\Sigma$  and includes $\sharp\Fr(I_1)-1$ frontier-arcs from polygons lying outside $\Sigma$ and distinct from $H_1$ and $H_3$.
\end{lemma}
\begin{proof}
We use the same notation as depicted in Figure~\ref{Fig10}. On one hand, by facts \ref{iii} and \ref{vi} we deduce that $\pi^{-1}\phi(J_3)=I_1$. Moreover by \ref{i}, $\pi(I_1) \subseteq \Sigma \cup H_1$. On the other hand, by Fact \ref{ii}, $\phi(J_3) \subseteq \Sigma \cup H_3$. So we deduce that $\pi(I_1)=\phi(J_3) \subseteq \Sigma$. The polygon $I_1$ is adjacent to $\Gamma$, more precisely, it shares one frontier-arc with $\Gamma$ and the remaining frontier-arcs are shared with polygons lying outside $\Gamma \cup I_2$. Then, $\pi(I_1)$ is not necessarily a polygon of the rigid structure. However, since $\pi$ is rigid outside $\Gamma \cup I_1$, it respects the adjacency of polygons outside $\Gamma \cup I_1$. Hence, among the $\sharp\Fr(I_1)$ images of the frontier-arcs of $I_1$ by $\pi$, $\sharp\Fr(I_1)-1$ are shared with polygons lying outside $\pi(\Gamma \cup I_2)= \Sigma \cup H_3$, whereas one image of a frontier arc is shared with $\pi(\Gamma)\subseteq \Sigma \cup H_1$. 
\end{proof}

Consider a $3$-corner with an attracting root in $\Dp(A_{n,m})$, as depicted in Figure~\ref{Fig10}. By the previous proof, $\pi(I_1)\subseteq \Sigma$ so we deduce that $\Sigma$ is an admissible surface containing at least one puncture: $h(\Sigma)\geq 1$.  

\begin{cor}\label{corHEightDanm}
    Let $x$ be a vertex with minimal height in a $3$-corner with an attracting root in $\Dp(A_{n,m})$. Then, $h(x)\geq 1$.
\end{cor}

\begin{remark}
    Let $\Sigma$ be an admissible surface in $\Cp(A_{n,m})$ and let $k$ be the height of $\Sigma$. If $\Sigma$ contains the central polygon then  the number of frontier-arcs of $\Sigma$ is given by $\sharp\Fr(\Sigma)=m+(n-1)(k-1)$, otherwise $ \sharp\Fr(\Sigma)=n+1+(n-1)(k-1)=k(n-1)+2$.
\end{remark}

\begin{lemma}\label{claimheight}
       Let $1 \leq m \leq n+1$ and let $x$ be a vertex with minimal height in a $3$-corner with an attracting root in $\Cp(A_{n,m})$. Then $h(x)\geq 2$. 
\end{lemma}

\begin{proof}
    Consider a $3$-corner with an attracting root in $\Cp(A_{n,m})$ as depicted in Figure~\ref{Fig10}, we use the same notations. By Lemma \ref{fact1}, $\pi(I_1)$ lies in $\Sigma$. Therefore $\Sigma$ has enough frontier-arcs so that $\pi(I_1)$ shares its boundary with $\sharp\Fr(I_1)-1$ polygons lying outside $\Sigma \cup H_1$ and distinct from $H_3=\pi(I_2)$ (since $I_1$ and $I_2$ are not adjacent). Consequently, the number of frontier arcs of $\Sigma$ minus two (for the one shared with $H_1$ and $H_3$) must be greater than (or equal to) $\sharp\Fr(I_1)-1$: \begin{equation}\label{bigstar}
        \sharp\Fr(\Sigma)-2 \geq \sharp\Fr(I_1)-1.
    \end{equation}
First, we assume that $I_1$ is not the central polygon $M$ and let $k:=h(\Sigma)=h(\Gamma)=h(\Delta)$. Hence (\ref{bigstar}) becomes: $$
 \left.
    \begin{array}{ll}
         m+(k-1)(n-1)-2 & \mbox{if } M\subset \Sigma \\
         k(n-1) & \mbox{otherwise}
    \end{array}
\right \}\geq n.
$$ Both cases imply directly that $k\geq 2$. 

Secondly, we assume that $I_1$ is the central polygon. Observe that if $m=n+1$, inequality (\ref{bigstar}) directly implies that $k\geq 2$. Then, by Fact \ref{i}, we have the following equality: 
\begin{equation}\label{bigstar2}
    \sharp\Fr(\Gamma \cup I_1)=\sharp\Fr(\Sigma \cup H_1).
\end{equation}
Under our assumptions this equality becomes:$$ 
m+k(n-1)=\left\{
    \begin{array}{ll}
         m+k(n-1) & \mbox{if } M\subseteq \Sigma \cup H_1 \\
         (k+1)(n-1)+2 & \mbox{otherwise}
    \end{array}
\right.$$ So if $M \not\subset \Sigma \cup H_1$, we obtain $m=n+1$ and so $k\geq 2$. We consider the remaining case where $I_1=M$ and $\Sigma \cup H_1$ contains $M$.\begin{itemize}
    \item \underline{Assume $H_1=M$}, and so $\Sigma$ does not contain $M$. By \ref{ii} and \ref{v}, $\phi(J_2)=H_1$ and by the rigidity of $\phi$ outside $\Delta \cup J_3$ we deduce that $J_2=M$. Now, since $I_1=M$, the subsurface $\Gamma \cup I_2$ does not contain $M$. Moreover by \ref{ii} one has $\sharp\Fr(\Gamma \cup I_2)=\sharp\Fr(\Delta \cup J_2),$ which becomes: $$(k+1)(n-1)+2=m+k(n-1)$$ so $m=n+1$ and $k\geq 2$.
    \item \underline{Assume $\Sigma$ contains $M$.} By \ref{iii} and \ref{vi}, $\pi^{-1}\phi(J_3)=I_1$ and by the rigidity of $\pi^{-1}\phi$ outside $\Delta \cup J_2$ we deduce that $J_3=M$. By contradiction, assume that $k=1$ and thus $\Sigma=M$. Recall that by Lemma \ref{fact1} one has $\pi(I_1)\subseteq \Sigma$, on top of that by \ref{i} $\pi$ is rigid outside $\Gamma \cup H_1$. Similarly, one has $\phi(J_3)\subseteq \Sigma$ and by \ref{v} $\phi$ is rigid outside $\Delta \cup J_3$. Hence, on one side $\pi(I_1)=M$ and $\phi(J_3)=M$. But on the other side this implies that $\pi(I_1)$ is adjacent to $H_3=\pi(I_2)$ and similarly, $\phi(J_3)$ is adjacent to $H_1=\phi(J_2)$, which contradicts the rigidity of $\pi$ and $\phi$ at $I_2$ and $J_2$, respectively. We conclude that  $k \geq 2.$ 
\end{itemize}

\end{proof}

Let $\amod^\square(A_{n,m})$ denote either $\amod(A_{n,m})$ or $\amod^*(A_{n,m})$.
\begin{lemma}\label{lemmaextrem}
    Let $\Sigma$ be an admissible surface in $\Cp(A_{n,m})$, resp. $\Dp(A_{n,m})$, with $h(\Sigma)\geq 2$. Let $D \subseteq \Sigma$ be a one-punctured disk such that: \begin{itemize}[-]
        \item $\Sigma\smallsetminus D$ is connected and $D$ contains exactly either $n$ or $m-1$ frontier-arcs from the polygons of the rigid structure;
        \item if $D$ contains $m-1$ frontiers arcs then either $m=n+1$ or $\Sigma$ contains the central polygon.
    \end{itemize} Then, there exist $\chi \in \amod^\square(A_{n,m})$ and an admissible surface $\Omega$ such that $[\Omega, \chi]=[\Sigma, \id]$ and $\chi^{-1}(D)\subseteq \Omega$ is a polygon from the rigid structure.
\end{lemma}

\begin{proof}
Let $\Sigma$ be an admissible surface in $\Cp(A_{n,m})$ (the proof is the same for $\Dp(A_{n,m})$), and $D \subseteq \Sigma$ be a one-punctured disk. We assume that the assumptions in the statement above hold. Let $K$ be a polygon inside $\Sigma$ with as many arcs from the rigid structure as $D$. There exists a minimal admissible surface $\Gamma$ containing both $D$ and $K$ and an asymptotically rigid homeomorphism $\mu$ permuting cyclically the frontier arcs of $\Gamma$ and such that the $n$ (resp. $m-1$) frontiers arcs contained in $D$ are sent onto those of $K$, see Figure \ref{Figures Article-49}. Then $\chi:=\mu^{-1}$ satisfies the statement with $\Omega:=\mu(\Sigma)$. \begin{figure}
    \centering
    \includegraphics[scale=0.16]{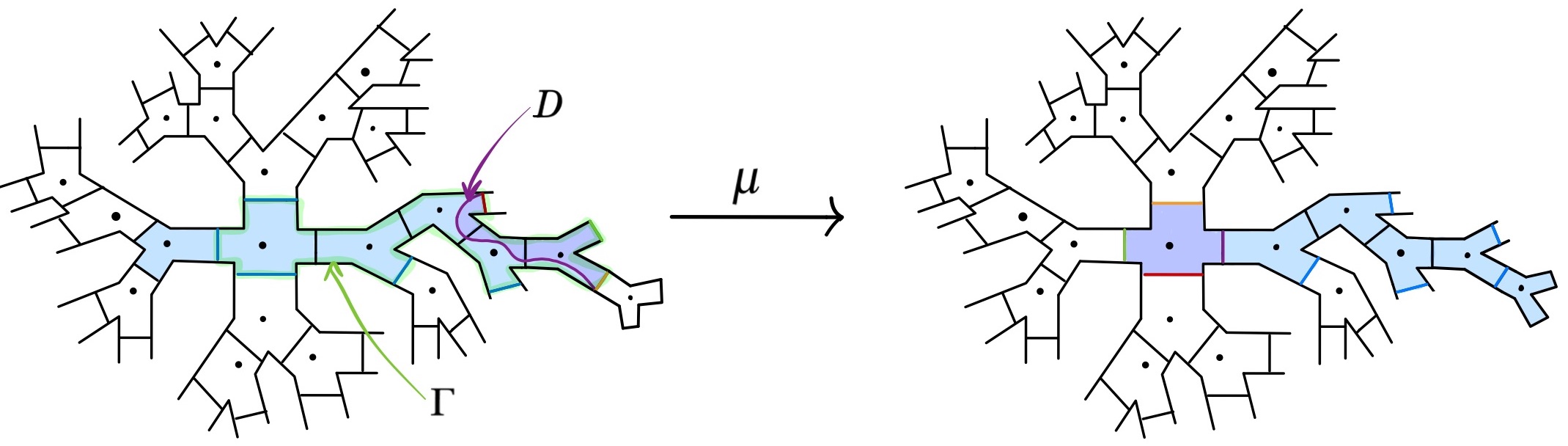}
    \caption{ }
    \label{Figures Article-49}
\end{figure}
\end{proof}

\begin{lemma}\label{height2}
    Consider a 3-corner with an attractive root whose minimum vertex height is $\geq 2$. Assume that we are in one of the following situations: \begin{itemize}[-]
           \item $1 \leq m \leq n+1$, and we are considering a 3-corner with an attracting root in $\Cp(A_{n,m})$;
           \item $(m,n)\neq (1,2),\; (2,1)$, and we are considering a 3-corner with an attracting root in $\Dp(A_{n,m})$.
       \end{itemize} 
       Then, the attracting root can be completed into a 3-cube.
\end{lemma}

\begin{proof}
    Consider a $3$-corner with an attracting root and the notations as in Figure \ref{Fig10}, assume that we are in one of the situations described by the assumptions of the statement, in particular $h(\Sigma)\geq 2$. Let $D=\pi(I_1)$, first we show that $D$ satisfies the assumptions of Lemma~\ref{lemmaextrem}. \begin{itemize}[-]
        \item Recall that in the case of a $3$-corner with an attracting root $\pi(\Gamma \cup I_ 1)= \Sigma \cup H_1$ where $I_ 1$ is adjacent to $\Gamma$ and $\pi(I_1)\subseteq \Sigma$. Thus $\Sigma\smallsetminus D$ is connected. Moreover, $I_1$ shares either $n$ or $m-1$ frontier-arcs with polygons outside $\Gamma \cup I_2$. By rigidity of $\pi$ outside $\Gamma \cup I_1$, we deduce that $D=\pi(I_1)$ contains $n$ or $m-1$ frontier-arcs from the polygons of the rigid structure.
        
        \item Assume that $D$ contains $m$ frontier-arcs. If $m=n+1$ the second assumption is satisfied. Otherwise $m\neq n+1$ and $I_1$ is the central polygon $M$. Recall that in our situation one has: $\sharp\Fr(\Gamma \cup I_1)=\sharp\Fr(\Sigma \cup H_1).$ On one side $\sharp\Fr(\Gamma \cup I_1)=m+k(n-1)$. On the other side, if $M$ is not contained in $\Sigma \cup H_1$ then $\sharp\Fr(\Sigma \cup H_1)=(k+1)(n-1)+2$ and so $m=n+1$ which is excluded. Hence $M$ is contained in $\Sigma \cup H_1$, if $M$ is contained in $\Sigma$ then the assumption is satisfied. By contradiction if $H_1=M$ then $\Sigma$ does not contain $M$. By \ref{ii} and \ref{v}, $\phi(J_2)=H_1$ and by the rigidity of $\phi$ outside $\Delta \cup J_3$ we deduce that $J_2=M$. Now, since $I_1=M$, the subsurface $\Gamma \cup I_2$ does not contain $M$. Moreover by \ref{ii} one has $\sharp\Fr(\Gamma \cup I_2)=\sharp\Fr(\Delta \cup J_2),$ which becomes: $(k+1)(n-1)+2=m+k(n-1)$ so $m=n+1$ which is excluded.
    \end{itemize}
By Lemma~\ref{lemmaextrem} applied with $D=\pi(I_1)$, we can rewrite: $$[\Sigma \cup H_1 \cup H_3, \id]=[\Omega \cup H'_1 \cup H'_3, \chi]=[(\Omega\smallsetminus K) \cup K \cup H'_1 \cup H'_3, \chi],$$ where $K:=\chi^{-1}(D)$ is a polygon of the rigid structure and $H'_i=\chi^{-1}(H_i)$. Observe that $\Omega\smallsetminus K$ is connected. Thus, we complete our $3$-corner into a $3$-cube by adding the vertex $[\Omega\smallsetminus K, \chi]$. See Figure~\ref{Fig09}.
\begin{figure}
    \centering
    \includegraphics[scale=0.25]{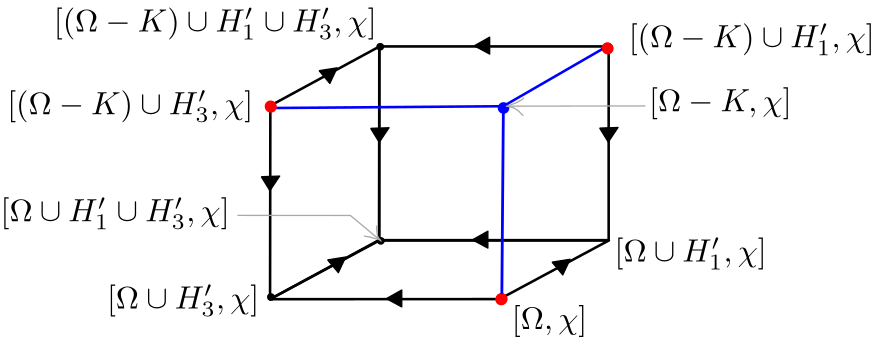}
    \caption{ }
    \label{Fig09}
\end{figure}

\end{proof}

\begin{proof}[Proof of Proposition \ref{prop3corner}]
Consider a $3$-corner with an attracting root and the notations as in Figure \ref{Fig10}. Let $1 \leq m \leq n+1$, by Lemma \ref{claimheight} in case of a $3$-corner with an attracting root in $\Cp(A_{n,m})$ one has $h(\Sigma)\geq2$ by Lemma \ref{height2} the $3$-corner can be completed in a $3$-cube. By the same lemma and Corollary \ref{corHEightDanm}, it remains to consider the $3$-corners in $\Dp(A_{n,m})$ with an attracting root such that $h(\Sigma)=1$. 

First assume that $m\geq 3$ and consider a $3$-corner with an attracting root in $\Dp(A_{n,m})$ such that $h(\Sigma)= 1$, then there exists a polygon, denoted as $H$, adjacent to the central polygon in such a way that $\Sigma = M \cup H$. If the $H_i$ are both adjacent to $M$ we are done since in this case the vertex $[M, \id]$ completes the $3$-corner into a $3$-cube. Otherwise, 
    \begin{itemize}
        \item either $H_1$ and $H_3$ are not adjacent to $M$, thus they are both adjacent to $H$ and lie in the $R_1$ strand. Since $m\geq 3$ one can push $R_1$ into $R_2$ using $\psi$ described in Figure \ref{Figure_12.2}. Then let $\chi:=\psi^{-1}$, $\Omega:=\psi(\Sigma)$ and $H'_i=\chi^{-1}(H_i)$. 
        \begin{figure} 
        \centering
        \includegraphics[scale=0.19]{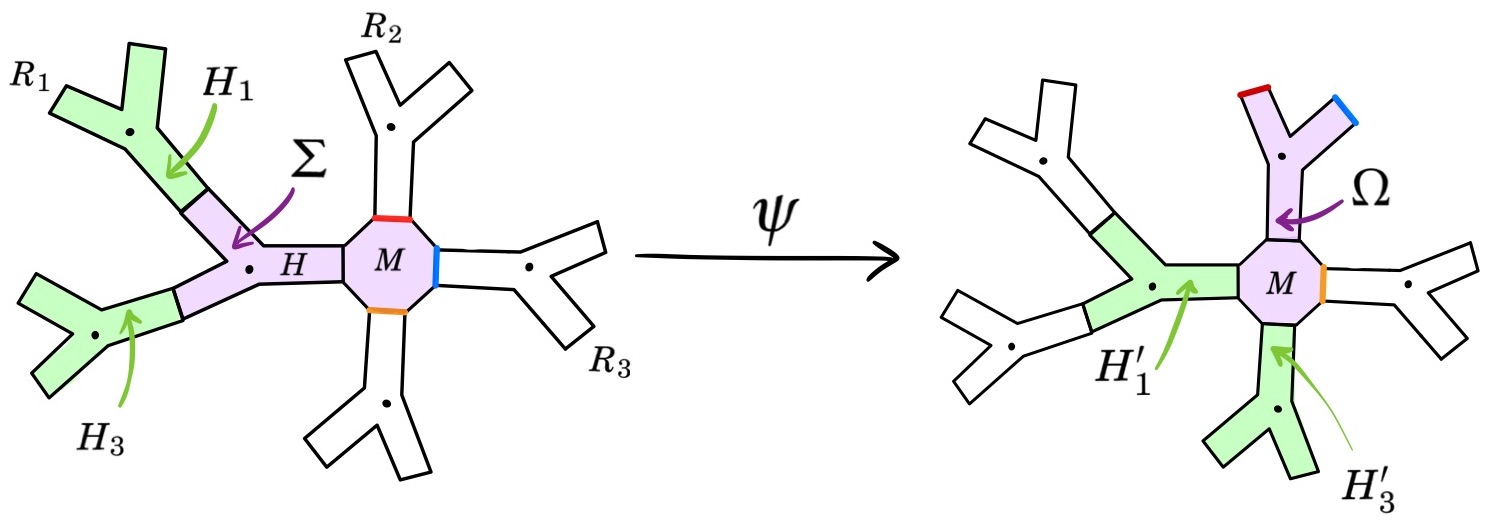}
        \caption{ }
        \label{Figure_12.2}
        \end{figure}
        \item or $H_1$ is adjacent to $M$ whereas $H_3$ is not. We proceed as in the previous case but we need to be careful to push the strands where $H_3$ lies in the one where $H_1$ does not lie (this is valid since here $m\geq$ 3).
        \end{itemize}
        In both cases we complete our $3$-corner into a $3$-cube by adding the vertex $[M, \chi]$.

        Secondly, assume that $m=1$, $n\geq 2$ and let $h(\Sigma)= 1$. Since $m=1$ one has $\Sigma=\Gamma=\Delta=M\cup H$ and $M$ shares no arcs with polygons distinct from $H$. By Fact \ref{i} the $n-1$ frontier arcs of $I_1$ (not shared with $H$) can only be sent to those of $H$ and those of $H_1$ (not shared with resp. $H_1$ and $H$). Thus there exists $\chi$ homotopic to $\pi$ such that $\chi(M)=M$. Hence we have $[M,\id]=[M,\pi]$. In the same way, we obtain $[M,\phi]=[M,\id]$, and so the latter vertex completes the $3$-corner into a $3$-cube.

        Thirdly, let $m=2$, $n\geq 2$ and let $h(\Sigma)= 1$. A priori, we have two possible configurations: either the $H_i$'s (resp. $I_i$'s, $J$'s) are not adjacent to $M$ or one of them is. Assume that, for instance $[\Sigma \cup H_1 \cup H_3, \id]$ has the second configuration. Then, since $m=2$ and $n\geq 2$, there exists $\psi \in \amod^*(A_{n,m})$ rigid outside $\Sigma$ which permutes cyclically the frontier arcs of $\Sigma$ so that $[\Omega \cup H_1' \cup H_3', \psi^{-1}]$ has the first configuration, where $\Omega:=\psi(\Sigma \cup H_1 \cup H_3)$. Hence, $[\Sigma \cup H_1 \cup H_3, \id]=[\Omega \cup H_1' \cup H_3', \psi^{-1}]$ where $\psi(H_i)=H_i'$ and so we reduce our work to the first configuration. In this case, with a reasoning similar to the case $m=1$ (excepting that $\Sigma=M\cup H$ may differ to $\Gamma=M\cup H'$), we complete the $3$-corner by adding $[M,\id]$. 
\end{proof}

\subsection{When is  \texorpdfstring{$\Cp(A_{n,m})$}{TEXT} a \texorpdfstring{$\CAT$}{TEXT} cube complex?} 
      \begin{figure}
    \centering
    \includegraphics[scale=0.2]{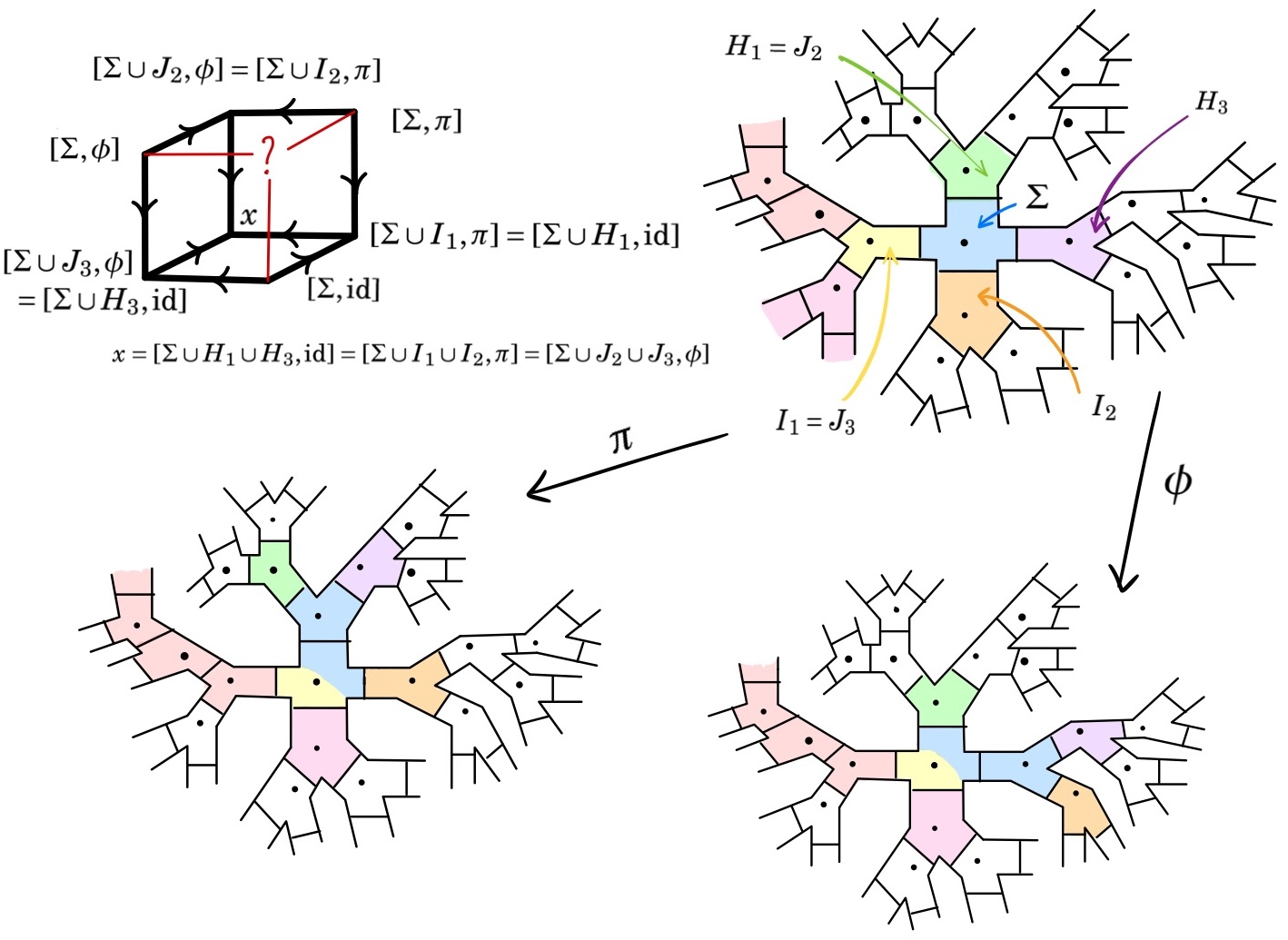}
    \caption{ }
    \label{Fig_2.12}
\end{figure}
 In this section, we determine for which sub-family of $A_{n,m}$ the associated cube $\Cp(A_{n,m})$ is $\CAT$. Let us first consider some examples.
 
\begin{example}\label{exA12}
  The cube complexes $\Cp(A_{1,1})$ and  $\Cp(A_{1,2})$ are $\CAT$. The two cases are similar so we consider $\Cp(A_{1,2})$. By Theorem 3.3 in \cite{asymI}, $\Cp(A_{1,2})$ is simply connected. By Lemma \ref{lemmaK23}, there is no copy of $K_{2,3}$. By definition we do not have $1$-corners and by Lemma \ref{lem3cornerA} it only remains to show that in $\Cp(A_{1,2})$ every $3$-corner with an attracting root can be completed into a $3$-cube. This follows from the fact that $\Cp(A_{1,2})$ is a square-complex (there are only two ways to remove or add adjacent polygons). Since $\pi(I_2)=H_3$ and $\pi$ is rigid outside $\Gamma \cup \I_1$ one cannot have $\pi(I_1) \subseteq \Sigma$. Thus, $\pi(I_1)=H_1$, $\pi(\Gamma)=\Sigma$ and so $[\Sigma, \id]=[\Gamma, \pi]$. Consequently, there are no $3$-corners. And we conclude that $\Cp(A_{1,2})$ is $\CAT$. Similarly one can show that $\Cp(A_{1,1})$, $\Dp(A_{1,1})$ and $\Dp(A_{1,2})$ are $\CAT$.
\end{example}

\begin{remark}\label{CA1M}
    As noted in \cite{asymI}, $\Cp(A_{1,3})$ is not $\CAT$, and this observation can be extended to $A_{1,m}$ for $m\geq 3$.
\end{remark}

Consider the cube complex $\Cp(A_{n,m})$ and denote by $R_1,\dots R_m$ the $m$ trees branching out the central polygon $M$. An \emph{infinite ray of polygons} $L$ in $R_j$ is a semi-infinite chain of polygons $L=(A_{j}^1,A_{j}^2,\dots)$ starting from the polygon adjacent to $M$ in $R_j$ and such that $A_{j}^k$ and $A_{j}^{k+1}$ are adjacent for all $k\geq1$.

\begin{example}\label{nonex2}
    For $m > n+1$ the cube complex $\Cp(A_{n,m})$ is not $\CAT$. The key obstacle here is that there is no vertices in $\Cp(A_{n,m})$ of height $0$. Consider the $3$-corner in $\Cp(A_{n,m})$ with the same notation as in Figure \ref{Fig10}. We construct a $3$-corner with $3$ vertices of height $1$, $3$ of height $2$, and one of height $3$ so that, by Corollary \ref{heightcube} the only way to complete it to a $3$-cube would be to add a vertex of height $0$. We take the following vertices: let $\Sigma$, $\Delta$ and $\Gamma$ be the central polygon $M$ and  let $I_1=J_3$ (resp. $H_1=J_2$, $H_3=I_2$) be the polygons in $R_1$ (resp. $R_2$, $R_3$) adjacent to the central polygon. See Figure \ref{Fig_2.12}.
    
    Fix three \emph{infinite rays} of polygons $L_1=(A_{1}^1,A_{1}^2,..)$, (resp. $L_2$, resp. $L_3$) in $R_1$ (resp. $R_2$, resp. $R_3$). The ray $L_i$ is built recursively as follows. Start by letting $A_{i}^1$ be the polygon adjacent to $M$ and lying in the branch $R_i$. Then as $A_{i}^2$ take the adjacent polygon whose puncture corresponds to the right-most child of the vertex corresponding to the puncture of $A_{i}^1$. Then we iterate the process. See Figure \ref{Figures Article-50} for an illustration of the infinite ray $L_1$.
    \begin{figure}[h]
    \centering
    \includegraphics[scale=0.1]{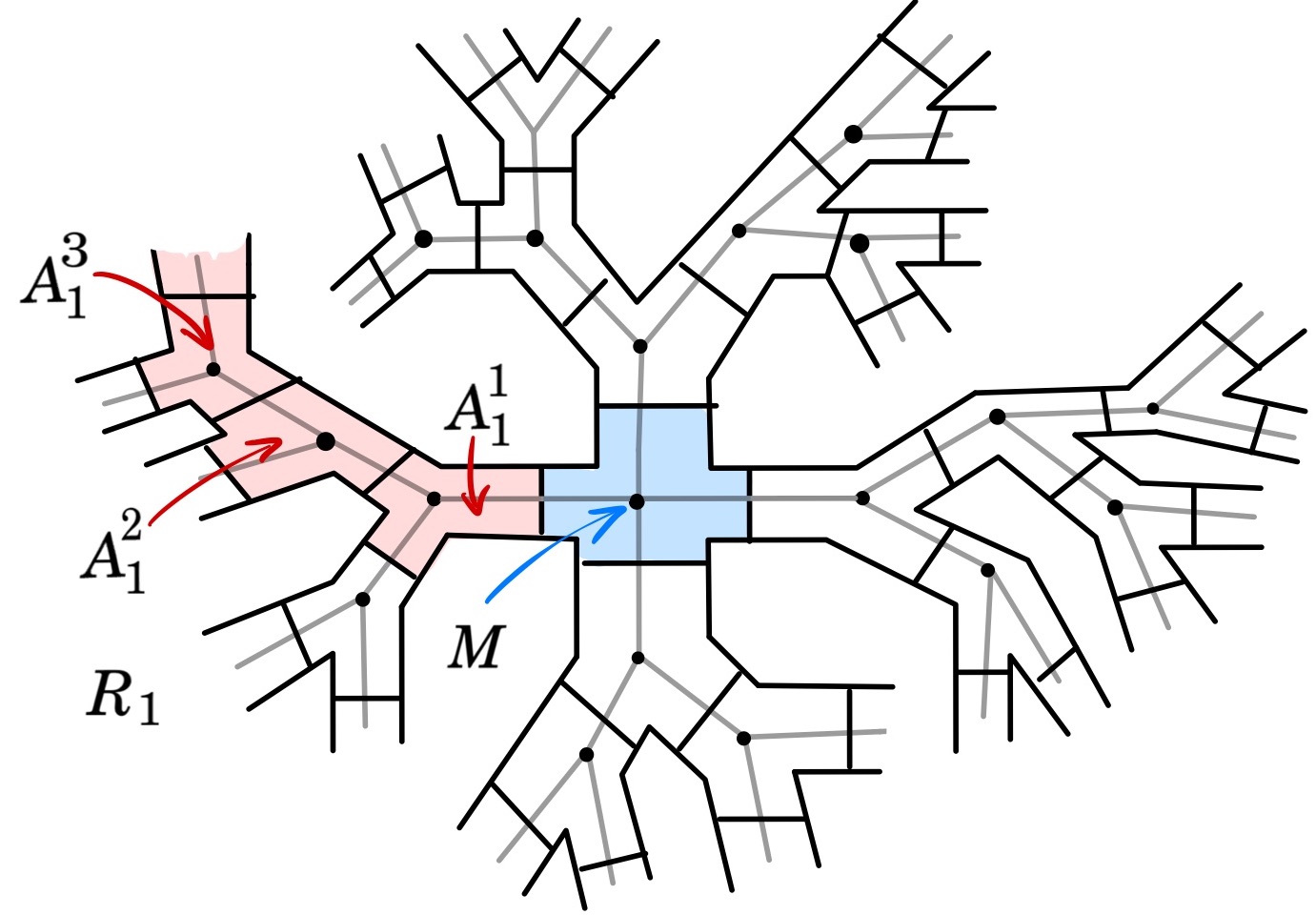}
    \caption{ }
    \label{Figures Article-50}
\end{figure} 
    Let $\pi$ be an asymptotically rigid homeomorphism that shifts each polygon of $L_1 \cup M \cup L_2$ by one polygon such that $\pi(A_{1}^1)$ lies in the central polygon. Let $\phi$ be an asymptotically rigid homeomorphism that shifts each polygon of $L_1 \cup M \cup L_3$ by one polygon such that $\phi(A_{1}^1)$ is the central polygon. See Figure \ref{Fig_2.12} for an illustration with $m=4$ and $n=2$.
    Then by construction, we cannot complete this $3$-corner into a $3$-cube because we do not have a vertex of height zero. So the link of the vertex $x=[\Sigma \cup H_1 \cup H_3, \id]$ represented in  Figure \ref{Fig_2.12} is not flag and $\Cp(A_{n,m})$ is not $\CAT$ for $m > n+1$
\end{example}

\begin{theorem}\label{Prop1}
    Let $n\geq 1$,  the cube complex $\Cp(A_{n,m})$ is $\CAT$ if and only if $1 \leq m \leq n+1$.
\end{theorem}

\begin{proof}
    Let $n\geq 1$ and $1 \leq m \leq n+1$. The case $n=1$ is given by Example~\ref{exA12}. By Theorem 3.3 in \cite{asymI}, $\Cp(A_{n,m})$ is contractible and thus simply connected. Proposition \ref{cubecomp} allows us to use Proposition \ref{Prop2} with $\Cp(A_{n,m})$ in order to prove Theorem \ref{Prop1}. Thus, we need to show that $\Cp(A_{n,m})$ satisfies the assumptions in Proposition \ref{Prop2}. 
     By definition we do not have $1$-corners and by Lemma \ref{lemmaK23} there is no $2$-corner. By Proposition \ref{prop3corner} every $3$-corners with an attracting root can be completed into a $3$-cube, the same holds for non-attracting root by Lemma \ref{lem3cornerA}. By Proposition \ref{Prop2} we conclude that $\Cp(A_{n,m})$ is $\CAT$.
    
    For the converse, Let $m>n+1$, if $n=1$ by Remark \ref{CA1M} $\Cp(A_{1,m})$ is not $\CAT$ for $m\geq 3$. Then if $n>1, m>n+1$, $\Cp(A_{n,m})$ is not $\CAT$ by Example \ref{nonex2}.

    This concludes the proof of Theorem \ref{Prop1}.
\end{proof}

\begin{remark}
    Let $m,n \geq 1$, the subcomplex of $\Cp(A_{n,m})$ generated by the vertices of the form $[\Sigma,\phi]$ where $\Sigma$ is an admissible surface containing the central polygon is called the \emph{spine}. This contractible complex was first introduce in \cite{asymI} and is denoted by $\Sp\Cp (A_{n,m})$. With a similar work we can show that $\Sp\Cp(A_{n,m})$ is $\CAT$ if and only if $n\geq 1$ and $1 \leq m \leq n+1$.
\end{remark}

\subsection{\texorpdfstring{\texorpdfstring{$\CAT$}{TEXT}ness of the cube complex $\Dp(A_{n,m})$}{TEXT}}
     
First, we show that $\Dp(A_{n,m})$ is contractible. To do this we follow closely the proof of the contractibility of $\Cp(A_{n,m})$ in Section 3.2 \cite{asymI}. Let $x$ be a vertex in $\Dp(A_{n,m})$ and $\Sigma \subseteq \Sp^*(A_{n,m})$ be an admissible surface. If there exists a finite path of vertices $x_0,x_1,\dots x_n$ from $x_0=x$ to $x_n=[\Sigma, \id]$ such that $h(x_i)>h(x_{i-1})$, one says that $[\Sigma, \id]$ \emph{dominates} $x$. Let $\mathcal{S}$ be a finite collection of vertices in $\Dp(A_{n,m})$. We denote by $X(\mathcal{S},\Sigma)$ the subcomplex in $\Dp(A_{n,m})$ generated by the height-increasing paths from a vertex in $\mathcal{S}$ to $[\Sigma, \id]$.

\begin{lemma}[\cite{asymI} Claim 3.6]\label{lemmadom}
    Let $\mathcal{S}$ be a finite collection of vertices in $\Dp(A_{n,m})$. Then, there exists an admissible surface $\Sigma$ such that $[\Sigma, \id]$ dominates all the vertices in $\mathcal{S}$.
\end{lemma}

\begin{proof}
    Let $\mathcal{S}=([\Delta_i, \phi_i])_{i \in I}$ and let $\Sigma_i$ be the admissible surface supporting $\phi_i$ for each $i \in I$. we denote by $\Omega_i$ the smallest admissible surface such that $\Sigma_i \cup \Delta_i \subseteq \Omega_i$, for each $i \in I$. Then let $\Sigma$ be an admissible surface containing all the $\phi(\Omega_i)$. One obtains a height-increasing path from $[\Delta_i, \phi_i]$ to $[\Omega_i, \phi_i]$ by adding polygons. Similarly, by definition of $\Sigma$, one has a second height-increasing path from $[\Omega_i, \phi_i]$ to $[\Sigma, \id]$. By concatenating the two paths we obtain a height-increasing path from $[\Delta_i, \phi_i]$ to $[\Sigma, \id]$. Thus $[\Sigma, \id]$ dominates all the vertices in $\mathcal{S}$.
\end{proof}

For the next lemma, the proof is the same steps as in Section 3.2 of \cite{asymI}, where the case of $\Cp(A_{n,m})$ is considered. We review the proof. 

\begin{lemma}\label{lemmaXcontr}
    Let $\mathcal{S}$ be a finite collection of vertices in $\Dp(A_{n,m})$. Then, there exists an admissible surface $\Sigma$ such that $X(\mathcal{S},\Sigma)$ is contractible.
\end{lemma}

\begin{proof}
     Let $\mathcal{S}$ be a finite collection of vertices in $\Dp(A_{n,m})$. By Lemma \ref{lemmadom} there exists an admissible surface $\Sigma$ such that $[\Sigma, \id]$ dominates all the vertices in $\mathcal{S}$. If $X(\mathcal{S},\Sigma)=[\Sigma, \id]$ we are done. Assume that $X(\mathcal{S},\Sigma) \neq [\Sigma, \id]$. Now, take a vertex $x \in X(\mathcal{S},\Sigma)$ with minimal height. Since $X(\mathcal{S},\Sigma) \neq [\Sigma, \id]$, $x$ admits $x_1,\dots,x_k$ as neighbours in $X(\mathcal{S},\Sigma)$. By Lemma \ref{lemma3.4} we assume that $x=[\Delta, \phi]$ and $x_i= [\Delta \cup H_i, \phi]$ for $1 \leq i \leq k$ where $\Delta$ is an admissible surface and $H_1,\dots,H_k$ are some adjacent polygons. 
     \begin{claim}
         The cube generated by the vertices $\{ [\Delta \cup_{i \in I} H_i, \phi] \}_{I \subseteq \{1,\dots,k\}}$ lies in $X(\mathcal{S},\Sigma)$
     \end{claim}

     \begin{proof}
         One has a height-increasing path from $x$ to $[\Sigma, \id]$. Thus by Lemma \ref{lemma3.4} $[\Sigma, \id]=[\Delta \cup J_1 \cup \dots \cup J_p, \phi]$ where $J_1,\dots,J_p$ are some polygons. Similarly since one also has an increasing path from $x_i=[\Delta \cup H_i, \phi]$ to $[\Sigma, \id]$, by Lemma \ref{lemma3.4}, $[\Sigma, \id]=[\Delta \cup H_i \cup M^i_2 \cup \dots \cup M^i_l, \phi]$ where $M^i_2,\dots,M^i_l$ are some polygons. To sum up: $$[\Sigma, \id]=[\Delta \cup J_1 \cup \dots \cup J_p, \phi]=[\Delta \cup H_i \cup M^i_1 \cup \dots \cup M^i_l, \phi],$$ so $l=p$ and $J_i=H_{\sigma(i)}$ for some permutation $\sigma$. Hence, one can add to $\Delta$ the adjacent polygons $H_i$ and next the remaining $J_i$  to obtain a path from $x$ to $[\Sigma, \id ]$ passing through the $\{ [\Delta \cup_{i \in I} H_i | I \subseteq \{1,\dots,k\} ] \}$. 
     \end{proof}

    In particular, all the direct neighbours $x_i$'s of $x$ in $X(\mathcal{S},\Sigma)$ are contained in a cube in $X(\mathcal{S},\Sigma)$. Thus the complex $X(\mathcal{S},\Sigma)$ deformation retracts onto $X((\mathcal{S}\smallsetminus\{x\})\cup \{x_1,\dots,x_k\},\Sigma)$. Since $[\Sigma, \id]$ dominates the vertices in $(\mathcal{S}\smallsetminus \{x\})\cup \{x_1,\dots,x_k\}$ we iterate the process. Hence we find a sequence: $$X(\mathcal{S},\Sigma) \supset X(\mathcal{S}_1,\Sigma) \supset \dots. $$
    This sequence stops since $X(\mathcal{S},\Sigma)$ contains finitely many cells. Hence we obtain from a certain rank $k$, $X(\mathcal{S}_k,\Sigma)=[\Sigma, \id]$ for some finite collection $\mathcal{S}_k$ dominated by $[\Sigma, \id]$. Thus $X(\mathcal{S},\Sigma)$ deformation retracts onto $[\Sigma, \id]$ so is contractible
\end{proof}

\begin{proposition}\label{Dcontr}
    Let $m,n\geq 1$ the cube complexe $\Dp(A_{n,m})$ is contractible.
\end{proposition}
\begin{proof}
     Let $n\geq 1$ and $\psi:\Sph^n \longrightarrow \Dp(A_{n,m})$ be a continuous map. The image of $\psi$ lies in compact subcomplex which has a finite collection of vertices $\mathcal{S}$. By Lemma \ref{lemmadom}, there exists an admissible surface $\Sigma$ such that $[\Sigma, \id]$ dominates all the vertices in $\mathcal{S}$. Thus the image of $\psi$ lies in $X(\mathcal{S},\Sigma)$ which is contractible by Lemma \ref{lemmaXcontr}. Hence $\psi$ is homotopically trivial. By Whitehead's theorem we deduce that $\Dp(A_{n,m})$ is contractible. 
\end{proof}

\begin{proposition}\label{DpisCAT0}
     Let $m, n\geq 1$ then the cube complex $\Dp(A_{n,m})$ is $\CAT$.
\end{proposition}

\begin{proof}
Let $n,m\geq 1$. By Proposition \ref{Dcontr} $\Dp(A_{n,m})$ is contractible and thus simply connected. Proposition \ref{cubecomp} allows us to use Proposition \ref{Prop2} with $\Dp(A_{n,m})$ in order to show that $\Dp(A_{n,m})$ is $\CAT$. Thus, we need to show that $\Dp(A_{n,m})$ satisfies the assumptions in Proposition \ref{Prop2}. By definition we do not have $1$-corners and by Lemma \ref{lemmaK23} there is no $2$-corner. In Example \ref{exA12} we show that we don't have $3$-corners in $\Dp(A_{1,1})$ and $\Dp(A_{1,2})$ so by Proposition \ref{prop3corner} every $3$-corners with an attracting root can be completed into a $3$-cube. By Lemma \ref{lem3cornerA} every $3$-corners with a non-attracting root can be completed into a $3$-cube. Hence, by Proposition \ref{Prop2} we conclude that $\Dp(A_{n,m})$ is $\CAT$.
 \end{proof}

Now we define the collection of cube complexes,
$$
\mathcal{E}(A_{n,m}):= \left\{
    \begin{array}{ll}
        \Dp(A_{n,m-n+1}) & \mbox{if } m > n+1 \geq 1 \\
        \Cp(A_{n,m}) & \mbox{if } 1 \leq m \leq n+1
    \end{array}.
\right.
$$

\begin{cor}\label{EndCor}   
     For all $m,n \geq 1$, $\amod(A_{n,m})$ acts on the cube complex $\mathcal{E}(A_{n,m})$ which is $\CAT$.
\end{cor} 

\begin{proof}
    Let $m,n \geq 1$, if $1 \leq m \leq n+1$ then $\amod(A_{n,m})$ acts on $\mathcal{E}(A_{n,m})=\Cp(A_{n,m})$ which is $\CAT$ by Theorem \ref{Prop1}. Otherwise, $m > n+1$ and by Lemma \ref{LemmaIsoStruct} $\amod(A_{n,m})\cong \amod^*(A_{n,m-n+1})$, now the latter group acts on $\mathcal{E}(A_{n,m})= \Dp(A_{n,m-n+1})$ which is $\CAT$ by Theorem \ref{DpisCAT0}.
\end{proof}
 
\printbibliography
\vspace{10pt}
\end{document}